%% file: special.tex
\documentclass{amsart}

\usepackage{amscd}
\usepackage{amssymb}

\usepackage{color}

\newtheorem{lem}[equation]{Lemma}

\newtheorem{cor}[equation]{Corollary}
\newtheorem{prop}[equation]{Proposition}
\newtheorem{conj}[equation]{Conjecture}

\newtheorem{longroot}[equation]{Long Root Proposition}
\newtheorem{shortroot}[equation]{Short Root Proposition}

\newtheorem{specthm}{Theorem}

\newtheorem*{thm*}{Theorem}
\newtheorem*{prop*}{Proposition}
\newtheorem*{cor*}{Corollary}
\newtheorem*{lem*}{Lemma}
\newtheorem*{MT*}{Main Theorem}

\newtheorem*{ques*}{Question}

\theoremstyle{definition} %

\newtheorem*{defn*}{Definition}

\newtheorem{eg}[equation]{Example}

\theoremstyle{remark} %
\newtheorem{rmk}[equation]{Remark}

\newtheorem*{rmk*}{Remark}
\newtheorem*{rmks*}{Remarks}

\DeclareMathOperator{\rank}{rank}
\DeclareMathOperator{\car}{char}
\DeclareMathOperator{\Lie}{Lie}

\DeclareMathOperator{\im}{im}

\newcommand{\C}{\mathbb{C}}
\newcommand{\Z}{\mathbb{Z}}

\newcommand{\gn}{\mathfrak{n}}
\newcommand{\gd}{\mathfrak{d}}

\newcommand{\glong}{\mathfrak{\g_{>}}}
\newcommand{\Glong}{G_{>}}

\newcommand{\gshort}{\mathfrak{\g_{<}}}
\newcommand{\Gshort}{G_{<}}

\newcommand{\lie}{\mathfrak{g}}
\newcommand{\lsub}{\mathfrak{h}}

\newcommand{\m}{\mathfrak{m}}

\newcommand{\qform}[1]{{\left\langle{#1}\right\rangle}}                   

\DeclareMathOperator{\SL}{SL}
\DeclareMathOperator{\Sp}{Sp}

\DeclareMathOperator{\SO}{SO}
\DeclareMathOperator{\PSp}{PSp}
\newcommand{\psp}{\mathfrak{psp}}
\DeclareMathOperator{\GSp}{GSp}
\newcommand{\gsp}{\mathfrak{gsp}}

\DeclareMathOperator{\GL}{GL}
\DeclareMathOperator{\Ad}{Ad}

\DeclareMathOperator{\Spin}{Spin}
\DeclareMathOperator{\End}{End}

\newcommand{\g}{\mathfrak{g}}
\newcommand{\gt}{\tilde{\g}}

\newcommand{\n}{\mathfrak{n}}
\newcommand{\gl}{\mathfrak{gl}}
\renewcommand{\sl}{\mathfrak{sl}}

\newcommand{\pgl}{\mathfrak{pgl}}

\renewcommand{\sp}{\mathfrak{sp}}
\newcommand{\spin}{\mathfrak{spin}}

\newcommand{\so}{\mathfrak{so}}

\DeclareMathOperator{\Sym}{S}

\newcommand{\z}{\mathfrak{z}}

\newcommand{\Gm}{\mathbb{G}_m}

\newcommand{\drho}{\mathrm{d}\rho}

\newcommand{\dpi}{\mathrm{d}\pi}

\newcommand{\la}{\lambda}
\newcommand{\ot}{\otimes}

\newcommand{\eps}{\varepsilon}

\newcommand{\bx}{\bar{x}}

\usepackage{hyperref}
\hypersetup{
  colorlinks=true, 
   linkcolor=red,  
}
\usepackage[hyphenbreaks]{breakurl}

\numberwithin{equation}{section}

\begin{document}

\title[Generically free representations: extremely bad characteristic]{Generically free representations III: \\extremely bad characteristic}

\begin{abstract}
In parts I and II, we determined which faithful irreducible representations $V$ of a simple linear algebraic group $G$ are generically free for $\Lie(G)$, i.e., which $V$ have an open subset consisting of vectors whose stabilizer in $\Lie(G)$ is zero, with some assumptions on the characteristic of the field.  This paper settles the remaining cases, which are of a different nature because $\Lie(G)$ has a more complicated structure and there need not exist general dimension bounds of the sort that exist in good characteristic. 
\end{abstract}

\author[S. Garibaldi]{Skip Garibaldi}

\author[R.M. Guralnick]{Robert M. Guralnick}


%
\maketitle

Let $G$ be a simple algebraic group over an algebraically closed field $k$.  In case $k = \C$, it has been known for more than 40 years which irreducible representations $V$ of $G$ are \emph{generically free}, i.e., have the property that the stabilizer in $G$ of a generic $v \in V$ is the trivial group scheme.  Recent applications of this to the theory of essential dimension have motivated the desire to extend these results to arbitrary $k$.
We did this in previous papers --- \cite{GG:spin}, \cite{GurLawther}, \cite{GG:large}, and \cite{GG:irred} --- except for a handful of cases that we address here, completing the solution to the problem.  In particular we prove the following, which was announced at the end of \cite{GG:large}.   

\begin{specthm} \label{MT.group}
Let $\rho \!: G \to \GL(V)$ be a faithful irreducible representation of a simple algebraic group over an algebraically closed field $k$.
\begin{enumerate}
\item \label{MT.group.et} $G_v$ is finite \'etale for generic $v \in V$ if and only if $\dim V > \dim G$ and $(G, V)$ does not appear in Table \ref{irred.nvfree}.
\item \label{MT.group.free} $G$ acts generically freely on $V$ if and only if $\dim V > \dim G$ and $(G, V)$ appears in neither Table \ref{irred.nvfree} nor Table \ref{ngenfree}.
\end{enumerate}
\end{specthm}

We say that $V$ is \emph{faithful} if $\ker \rho$ is the trivial group scheme.  Every generically free representation is faithful.  The hypothesis that $\rho$ is faithful in Theorem \ref{MT.group} excludes those representations that factor through a special isogeny of $G$; nonetheless, we do consider such representations in detail in this paper.

\input ngenfree-lie-table-enhanced

\input ngenfree-table

(Recall that an algebraic group $H$ is \emph{finite \'etale} if $\dim H = 0$ and $\Lie(H) = 0$.)

The remaining cases of the theorem that need to be covered in this paper are those where $\car k$ is \emph{special}\footnote{This choice of vocabulary imitates \cite{St:rep}; we have written instead the more illuminating ``extremely bad characteristic'' in the title.  The hypothesis ``$\car k$ special'' is properly more restrictive than ``$\car k$ very bad'', in that 2 is very bad but not special (i.e., not extremely bad) for type $G_2$.}  for $G$, meaning that $G$ has type $G_2$ and $\car k = 3$ or $G$ has type $B_n$ ($n \ge 2$), $C_n$ ($n \ge 2$), or $F_4$ and $\car k = 2$.  These are the cases where the Dynkin diagram of $G$ has a multiple bond of valence $\car k$.  Equivalently, these are the cases where $G$ possesses so called ``special'' isogenies, which are neither central nor the Frobenius, cf.~\cite[\S3]{BoTi:hom}.

In a future work, we combine Theorem \ref{MT.group} with the results of \cite{GurLawther} to prove the existence of a stabilizer in general position for every action of a simple algebraic group on an irreducible representation.

A special case of Theorem \ref{MT.group} is the following:

\begin{specthm} \label{MT.special}
Let $G$ be a simple linear algebraic group over an algebraically closed field $k$ such that $\car k$ is special, and let $\rho \!: G \to \GL(V)$ be an irreducible and faithful representation.  Then $V$ is generically free for $\g$ if and only if $\dim V > \dim G$.
\end{specthm}


\subsection*{Large, possibly reducible representations}
In \cite{GG:large}, we proved a general bound when $G$ is simple and $\car k$ is not special: if $V^{[\g,\g]} = 0$ and $\dim V$ is big enough, then $\g$ acts virtually freely on $V$.  
However, Example \ref{large.eg} shows that such a result does not hold when $\car k$ is special.  Because the true results are of a varied nature, we do not include a summary statement here; see 
Proposition \ref{FG.dim}, Corollary \ref{B.quo},  and Proposition \ref{C.faithful} for precise statements.
Note that these results have no requirements that $G$ acts irreducibly or faithfully.  Roughly speaking, for $\n$ the Lie algebra of the kernel of the very special isogeny as in \S\ref{struct}, we give results for those $V$ on which $\n$ acts as zero ($\n V = 0$) or without fixed points ($V^\n = 0$).

\subsection*{Irreducible representations}
Recall that every irreducible representation $V$ of $G$ has a highest weight $\la$.  Write $\la$ as a sum $\la = \sum_\omega c_\omega \omega$ where the sum runs over the fundamental dominant weights $\omega$.  Then $\la$ is restricted if $p := \car k \ne 0$ and $0 \le c_\omega < p$ for all $\omega$. (In case $\car k = 0$, all dominant weights are, by definition, restricted.)  

Our next result is a variation on Theorem \ref{MT.special}, where we add the hypothesis that the highest weight of $\rho$ is restricted and drop the hypothesis that $\rho$ is faithful.  Regardless of whether $\rho$ is faithful, the stabilizer $\g_v$ of a generic $v \in V$ contains $\ker \drho$; we say that $\g$ acts \emph{virtually freely} on $V$ if $\g_v = \ker \drho$; this is the natural generalization of the notion of ``generically free'' to include the case where $\rho$ need not be faithful.

\begin{specthm} \label{MT.restricted}
Let $G$ be a simple linear algebraic group over an algebraically closed field $k$ such that $\car k$ is special for $G$.  Let $\rho \!: G \to \GL(V)$ be a nontrivial irreducible representation for $G$ with a restricted highest weight.  Then $\g$ acts virtually freely on $V$ if and only if $\dim V > \dim G$, except for those cases where $(G, V, k)$ appears in Table \ref{irred.vfree}.
\end{specthm}

\begin{table}[bth]
\begin{tabular}{cccrr} 
$G$&$\car k$&$V$&$\dim V$&$\dim \ker \drho$ \\ \hline\hline
$\Sp_6$&2&spin&8&14 \\
$\Sp_8$ (but not $\PSp_8$)&2&spin&16&27\\
$\Sp_{10}$&2&spin&32&44\\
$\Sp_{12}$ or $\PSp_{12}$&2&spin&64&65
\end{tabular}
\caption{The nontrivial restricted irreducible representations of simple $G$ with $\dim V \le \dim G$ that are virtually free for $\g$.} \label{irred.vfree}
\end{table}

We remark that, in the setting of Theorem \ref{MT.restricted} and on the level of abstract groups, $G_v(k)$ is always finite when $\dim V > \dim G$ by \cite{GurLawther}.

\smallskip

The organization of the paper is as follows.
We first (\S\ref{struct}) recall properties of $\g$ and the irreducible representations of $G$, focusing on the case of special characteristic.  A short section (\S\ref{back.sec}) then recalls results used to constrain the size of $\g_v$ for generic $v \in V$.  In sections \ref{long.sec} and \ref{short.sec}, we prove some results on generic stabilizers by leveraging the Lie algebras of the long and short root subgroups. The next several sections are devoted to groups by type, each under the assumption that $\car k$ is special: $F_4$ and $G_2$ in \S\ref{FG.sec}, $B_n$ in \S\ref{B.sec}, and $C_n$ in \S\ref{C.sec}.  In each section, we prove that, under various hypotheses on the representation $V$, if $\dim V$ is large enough, then $\g$ acts virtually freely on $V$.  We prove Theorem \ref{MT.restricted} for each group in its section.  The results based on $\dim V$ are far from uniform, so we provide in \S\ref{eg.sec} an example to show that the uniform result from \cite{GG:large} is false if one drops the hypothesis that $\car k$ is not special.  It remains to prove Theorem \ref{MT.special} in case the highest weight is not restricted, which we do in \S\ref{MT.special.sec}.  Finally, we prove Theorem \ref{MT.group} in \S\ref{MT.group.sec}.

\begin{table}[bth]
{\centering\noindent\makebox[450pt]{
\begin{tabular}[c]{p{2.2in}|p{3in}}

${(B_\ell)~~}$
\begin{picture}(7,2)(0,0)
\put(0,1){\circle*{3}}
\put(0,0){\line(1,0){20}}
\put(0,2){\line(1,0){20}}
\put(20,1){\circle*{3}}
\put(20,1){\line(1,0){20}}
\put(40,1){\circle*{3}}
\put(40,-1.6){ \mbox{$\cdots$}}
\put(62,1){\circle*{3}}
\put(62,1){\line(1,0){20}}
\put(82,1){\circle*{3}}
\put(82,1){\line(1,0){20}}
\put(9,-1){{\small\mbox{$<$}}}
\put(102,1){\circle*{3}}

\put(-2,-7){\mbox{\tiny $1$}}
\put(18,-7){\mbox{\tiny $2$}}
\put(38,-7){\mbox{\tiny $3$}}
\put(54,-7){\mbox{\tiny $\ell$$-$$2$}}
\put(75,-7){\mbox{\tiny $\ell$$-$$1$}}
\put(100,-7){\mbox{\tiny $\ell$}}
\end{picture}
\vspace{0.5cm}

&

${(C_\ell)~~}$
\begin{picture}(7,2)(0,0)
\put(0,1){\circle*{3}}
\put(0,0){\line(1,0){20}}
\put(0,2){\line(1,0){20}}
\put(20,1){\circle*{3}}
\put(20,1){\line(1,0){20}}
\put(40,1){\circle*{3}}
\put(40,-1.6){ \mbox{$\cdots$}}
\put(62,1){\circle*{3}}
\put(62,1){\line(1,0){20}}
\put(82,1){\circle*{3}}
\put(82,1){\line(1,0){20}}
\put(9,-1){{\small\mbox{$>$}}}
\put(102,1){\circle*{3}}

\put(-2,-7){\mbox{\tiny $1$}}
\put(18,-7){\mbox{\tiny $2$}}
\put(38,-7){\mbox{\tiny $3$}}
\put(54,-7){\mbox{\tiny $\ell$$-$$2$}}
\put(75,-7){\mbox{\tiny $\ell$$-$$1$}}
\put(100,-7){\mbox{\tiny $\ell$}}
\end{picture}
\vspace{0.5cm}
\\

${(F_4)~~}$
\begin{picture}(7,2)(0,0)
\put(0,1){\circle*{3}}
\put(0,1){\line(1,0){15}}
\put(15,1){\circle*{3}}
\put(15,0){\line(1,0){15}}
\put(15,2){\line(1,0){15}}
\put(19,-1){{\small\mbox{$>$}}}
\put(30,1){\circle*{3}}
\put(30,1){\line(1,0){15}}
\put(45,1){\circle*{3}}

\put(-2,-7){\mbox{\tiny $1$}}
\put(13,-7){\mbox{\tiny $2$}}
\put(28,-7){\mbox{\tiny $3$}}
\put(43,-7){\mbox{\tiny $4$}}
\end{picture} &
${(G_2)~~}$
\begin{picture}(7,2)(0,0)
\put(2,1){\circle*{3}}
\put(2,0.1){\line(1,0){15}}
\put(2,1.1){\line(1,0){15}}
\put(2,2.1){\line(1,0){15}}
\put(6,-1){{\small\mbox{$>$}}}
\put(17,1){\circle*{3}}

\put(0,-7){\mbox{\tiny $1$}}
\put(15,-7){\mbox{\tiny $2$}}
\end{picture}

\end{tabular}
}}
\caption{Dynkin diagrams of the non-simply-laced simple root systems, with simple roots numbered as in \cite{luebeck}.} \label{luebeck.table}
\end{table}

\subsection*{Acknowledgements} 
The referees deserve thanks for their detailed and helpful comments on earlier versions of this paper.
We thank Brian Conrad for his advice on group schemes.
Guralnick was partially supported by NSF grant DMS-1600056.

\section{Structure of $\g$ and $V$} \label{struct}

\subsection*{Structure of $\g$}
We refer to \cite{Hiss}, \cite{Hogeweij}, or \cite[\S1]{Pink:cpt} for properties of $\g := \Lie(G)$ when $G$ is simple. For example, when $G$ is simply connected,  we have: (1) \emph{$\g / \z(\g)$ is a reducible $G$-module if and only if $\car k$ is special}, and (2) \emph{$\g$ has a unique proper maximal $G$-submodule}, which we denote by $\m$.  Statement (2) can be seen by direct computation or because $\g$ is a Weyl module of $G$ in the sense of \cite{Jantzen}, the one whose highest weight is the highest root.

Supposing now that $\car k$ is special for $G$ and $\pi$ is the very special isogeny, we put $N := \ker \pi$ and $\gn := \ker \dpi$.  See \cite[\S7.1]{CGP2} or \cite[\S10]{St:rep} for a concrete description of $\pi$.  
(We note that $\gn$ is the subalgebra in $\g$ generated by the short root subalgebras, and it need not contain the center, e.g., in case $G = \Sp_{2\ell}$ with odd $\ell \ge 3$.)  As $N$ is normal in $G$, It follows that the subspace
\[
V^{\gn} := \{ v \in V \mid \text{$\drho(x)v = 0$ for all $x \in \gn$} \}
\]
is a $G$-invariant submodule of $V$.  

Examining the tables in \cite{Hiss} and \cite{Hogeweij}, we find the following:
\begin{lem} \label{g.submod}
Let $G$ be a simple and simply connected split algebraic group over a field $k$ whose characteristic is special for $G$.  If $\lsub$ is a $G$-invariant submodule of $\g$ then $\lsub \supseteq \gn$ or $\lsub \subseteq \z(\g)$. $\hfill\qed$
\end{lem}

\subsection*{Irreducible representations of $G$ when $\car k \ne 0$} 
Fix a pinning for $G$, which includes the data of a maximal torus $T$ and a choice of simple roots $\Delta$.  Then irreducible representations $\rho \!: G \to \GL(V)$ (up to equivalence) are in bijection with the set of dominant weights $\la \in T^*$, i.e., those $\la$ such that $\qform{ \la, \delta^\vee} \ge 0$ for all $\delta \in \Delta$.

Supposing now that $p := \car k \ne 0$.  Write $\la = \la_0 + p^r \la_1$, where $\la_0 = \sum_\omega c_\omega \omega$ and $0 \le c_\omega < p^r$ for all $\omega$.  If $\la_0$ and $p^{r-1} \la_1$ belong to $T^*$ (e.g., if $G$ is simply connected), then $L(\la) \cong L(\la_0) \ot L(p^{r-1} \la_1)^{[p]}$ \cite[II.3.16]{Jantzen}, the tensor product of $L(\la_0)$ and a Frobenius twist of $L(p^{r-1} \la_1)$.  As a representation of $\g$ (forgetting about the action of $G(k)$), this is the direct sum of $\dim L(\la_1)$ copies of $L(\la_0)$.

\subsection*{Irreducible representations of $G$ when $\car k$ is special}
Now suppose that $\car k$ is special for $G$, so in particular $\Delta$ has two root lengths.  Write a dominant weight $\la$ as  $\la = \sum c_\delta \omega_\delta$, where $c_\delta \ge 0$ and $\omega_\delta$ is the fundamental weight dual to $\delta^\vee$ for $\delta \in \Delta$.  We write $\la = \la_s + \la_\ell$ where $\la_s = \sum_{\text{$\delta$ short}} c_\delta \omega_\delta$ and $\la_\ell = \sum_{\text{$\delta$ long}} c_\delta \omega_\delta$, i.e., $\qform{\la_s, \delta^\vee} = 0$ for $\delta$ long and $\qform{\la_\ell, \delta^\vee} = 0$ for $\delta$ short.  Steinberg \cite{St:rep} shows that, when $G$ is simply connected, $L(\la) \cong L(\la_\ell) \ot L(\la_s)$ and that furthermore $L(\la_\ell)$ factors through the very special isogeny.

Suppose now that $\la$ is restricted.  Then $L(\la_s)$ is irreducible as a representation of $\gn$ \cite[p.~52]{St:rep}, so Lemma \ref{g.submod} shows that the kernel of this representation is contained in $\z(\g)$.  Similarly, as an $\gn$-module, $L(\la)$ is a direct sum of $\dim L(\la_\ell)$ copies of $L(\la_s)$, and again the kernel of the representation is contained in $\z(\g)$.

In summary, for $\la$ restricted and $G$ simply connected, we have either (1) $\la_s = 0$ and $\ker \drho \supseteq \gn$, or (2) $\la_s \ne 0$ and $\ker \drho \subseteq \z(\g)$.

\section{Lemmas for computing $\g_v$} \label{back.sec}

Choose a representation $\rho \!: G \to \GL(V)$.  For $x \in \lie$, put
\[
V^x := \{ v \in V \mid \drho(x)v = 0 \}
\]
and $x^G$ for the $G$-conjugacy class $\Ad(G)x$ of $x$.  Recall the following from \cite[\S1]{GG:large}.

\begin{lem} \label{ineq} 
For $x \in \g$, 
\begin{equation} \label{mother0}
x^G \cap \g_v = \emptyset \quad \text{for generic $v \in V$}
\end{equation}
is implied by:
\begin{equation} \label{ineq.mother}
  \dim x^G + \dim V^x < \dim V,
\end{equation}
which is implied by:
\begin{equation} \label{ineq.gen}
\parbox{4.25in}{{There exist $e > 0$ and $x_1, \ldots x_e \in x^G$ such that 
the subalgebra $\mathfrak{s}$ of $\g$ generated by $x_1, \ldots, x_e$ has $V^{\mathfrak{s}} = 0$ and
$e\cdot \dim x^G < \dim V$.}}
\end{equation}
\end{lem}

We use this as follows.  Choose a subalgebra $\lsub$ of $\g$.  To verify that $\g_v \subseteq \lsub$ for generic $v \in V$, it suffices to verify, for all $x \in \g \setminus \lsub$, that $x \not\in \g_v$, for which we may check \eqref{ineq.mother} or \eqref{ineq.gen}. 

For us, $\g = \Lie(G)$ and $\car k = p \ne 0$, so the Frobenius morphism on $G$ induces a $p$-operation $x \mapsto x^{[p]}$ on $\g$, see \cite{StradeF} for properties.  When $G$ is a sub-group-scheme of $\GL_n$ and $x \in \g$, the element $x^{[p]}$ is the $p$-th power of $x$ with respect to the typical, associative multiplication for $n$-by-$n$ matrices, see \cite[\S{II.7}, p.~274]{DG}.  

An element $x \in \g$ is \emph{nilpotent} if $x^{[p]^n} = 0$ for some $n > 0$, \emph{toral} if $x^{[p]} = x$, and \emph{semisimple} if $x$ is contained in the Lie $p$-subalgebra of $\g$ generated by $x^{[p]}$, cf.~\cite[\S2.3]{StradeF}.  We recall from part I:

\begin{lem} \label{ineq.spin}
Suppose $G$ is semisimple over an algebraically closed field $k$ of characteristic $p > 0$, and  let $\lsub$ be a $G$-invariant subspace of $\lie$.  
\begin{enumerate}
\item \label{mother.lie2} If inequality \eqref{ineq.mother} holds for every toral or nilpotent $x \in \g \setminus \lsub$, then  $\g_v \subseteq \lsub$ for generic $v \in V$.

\item \label{mother.lie} If $\lsub$ consists of semisimple elements and \eqref{ineq.mother} holds for every $x \in \g \setminus \lsub$ with $x^{[p]} \in \{ 0, x \}$, then $\g_v \subseteq \lsub$  for generic $v$ in $V$. $\hfill\qed$
\end{enumerate}
\end{lem}

One typical application of part \eqref{mother.lie} of the lemma is when $\lsub = \z(\g)$.


\section{A Heisenberg Lie algebra} \label{heis.sec}

Let $G = \Spin_{2n+1}$ over a field $k$ (always assumed algebraically closed) of characteristic 2.  The short root subalgebras of $\g$ generate a ``Heisenberg'' Lie algebra $\lsub$ of dimension $2n+1$ such that $[\lsub, \lsub]$ is the 1-dimensional center $\z(\lsub)$.  
The algebra $\lsub$ is the image of $\sl_2^{\times n}$ under a central isogeny $\SL_2^{\times n}  \to \SL_2^{\times n}/Z$ where $Z$ is isomorphic to $\mu_2^{\times (n-1)}$, and the quotient $\lsub/\z(\lsub)$ is the image of $\sl_2^{\times n} \to \pgl_2^{\times n}$.

For $G = \Sp_{2n}$ over the same $k$, the very special isogeny $\pi \!: \Sp_{2n} \to \Spin_{2n+1}$ has image $\lsub$, and so we may identify $\lsub$ with $\g / \ker \dpi$.

\begin{lem} \label{heis}
Suppose $\rho \!: G \to \GL(V)$ is a representation of $G = \Spin_{2n+1}$ or $\Sp_{2n}$.  In the latter case, assume additionally that $\drho$ vanishes on $\ker \dpi$.
\begin{enumerate}
\item \label{heis.1} If $4n + \dim V^{\z(\lsub)} < \dim V$, then $\dim x^G + \dim V^x < \dim V$ for all nonzero $x \in \lsub$.
\item \label{heis.gen} If $V^\lsub = 0$ and $4n^2 < \dim V$, then $\dim x^G + \dim V^x < \dim V$ for all noncentral $x \in \lsub$.
\end{enumerate}
\end{lem}

\begin{proof}
For $x$ nonzero central, $\dim x^G + \dim V^x = \dim V^{\z(\lsub)}$, verifying \eqref{heis.1}, so suppose $x$ is noncentral.

In case \eqref{heis.1}, there is a $g \in G(k)$ so that $[x, x^g]$ is nonzero central in $\lsub$, so $\dim V^x \le \frac12 (\dim V + \dim V^{\z(\lsub)})$.  As $\dim x^G < 2n + 1$, the claimed inequality follows.

In case \eqref{heis.gen}, the Weyl group of $G$ acts transitively on groups of a  given length, so $2n$ conjugates of $x$ generate $\lsub$, and therefore to prove the claim it suffices to note that $2n \cdot \dim x^G < \dim V$ \cite[Lemma 1.1]{GG:large}.
\end{proof}

\section{Long root subgroup} \label{long.sec}

Suppose $G$ is simple and simply connected and $\car k$ is special for $G$.  Fix a maximal torus $T$ in $G$.  The long root subgroups of $G$ (relative to $T$) generate a subgroup $\Glong$ that is also simply connected and the type of $(\Glong, G)$ is one of $(A_2, G_2)$, $(D_4, F_4)$, $(D_n, B_n)$ for $n \ge 2$, or $(A_1^n, C_n)$ for $n \ge 2$, \cite[Prop.~7.1.7]{CGP2}.  We put $\glong := \Lie(\Glong)$, which as a vector space is a direct sum of $\Lie(T)$ and the long root subalgebras.

\begin{lem}  \label{glong}
In the notation of the preceding paragraph, 
\begin{enumerate}
\item \label{glong.1} The composition $\glong \to \g \to \g/\gn$ is onto.
\item \label{glong.2} $\z(\glong) + \gn = \m$, the unique maximal $G$-invariant subspace of $\g$.
\item \label{glong.3} $\glong \cap \m = \z(\glong)$.
\end{enumerate}
\end{lem}

\begin{proof}
For \eqref{glong.1}, $\Lie(T)$ is contained in $\glong$, as is the long root subalgebra $\g_\alpha$ for every long root $\alpha$.  As $\gn$ contains $\g_\alpha$ for every short root $\alpha$, we have $\g = \glong + \gn$ as vector spaces.  Hence $\m = (\glong \cap \m) + \gn$, so \eqref{glong.3} implies \eqref{glong.2}.

For \eqref{glong.3}, the description of the $G$-submodules of $\g$ in \cite{Hiss} show that the long root subalgebras are not contained in $\m$.  As $\m$ and $\glong$ are $T$-invariant, we have $\m \cap \glong \subseteq \Lie(T)$.  For $z \in \glong \cap \m$ and $x_\alpha$ a root element in $\glong$, $[z,x_\alpha] \in \m$, so $z \in \z(\glong)$, i.e., \eqref{glong.3} holds.
\end{proof}

\begin{longroot} \label{glong.prop}
Let $G$ be a simple and simply connected algebraic group such that $\car k$ is special for $G$.  Suppose that $\rho \!: G \to \GL(V)$ vanishes on the kernel of the very special isogeny.  If $V$ has a subquotient $W$ such that $W^\g = 0$ and 
\[
\dim W > \begin{cases}
64 & \text{if $G = F_4$} \\
20 & \text{if $G = G_2$} \\
30 & \text{if $G = \Spin_7$} \\
4n^2 & \text{if $G = \Spin_{2n+1}$ with $n \ge 4$ or $G = \Sp_{2n}$ with $n \ge 2$,}
\end{cases}
\]
then $\g_v \subseteq \m$ for generic $v \in V$.
\end{longroot}

\begin{proof}
We verify, for $x \in \glong \setminus \z(\glong)$ such that $x^{[p]} \in \{0, x\}$, that
\begin{equation} \label{glong.ineq}
\dim x^{\Glong} + \dim W^x < \dim W.
\end{equation}

Note that $\Glong$ is not $\Sp_{2n}$ for any $n \ge 1$, so $[\glong, \glong] = \glong$.  Therefore, $W^{[\glong, \glong]} = W^\glong = W^\g = 0$.
For $G = G_2$, $F_4$, or $\Spin_{2n+1}$ for $n \ge 3$, we have \eqref{glong.ineq} by \cite[Th.~12.2]{GG:large}, where in case $\Spin_7$ we use the identity $\Spin_6 = \SL_4$. For $G = \Sp_{2n}$ with $n \ge 2$, we apply Lemma \ref{heis}\eqref{heis.gen}.

Because \eqref{glong.ineq} holds, we deduce that $\dim x^{\Glong} + \dim V^x < \dim V$ (elementary, see \cite[Example 2.1]{GG:large}).  As $\glong \cap \m = \z(\glong)$ consists of semisimple elements, it follows that $(\glong)_v \subset \m$ for generic $v \in V$, whence $\g_v \subset \m$.
\end{proof}

\section{Short root subgroup} \label{short.sec}

Continue the notation of the preceding section.  In particular, $G$ is assumed simply connected and $\car k$ is special for $G$.  The root subgroups in $G$ corresponding to short roots generate a subgroup $\Gshort$, and the type of $(\Gshort, G)$ is $(A_2, G_2)$, $(D_4, F_4)$, $(A_1^n, B_n)$ for $n \ge 2$, or $(D_n, C_n)$ for $n \ge 3$ \cite[Prop.~7.1.7]{CGP2}.  We put $\gshort := \Lie(\Gshort)$.

\begin{lem} \label{gshort}
$[\gshort, \gshort] = \gn$.
\end{lem}

\begin{proof}
Put $\widetilde{\Gshort}$ for the simply connected cover of $\Gshort$.
Because $\Gshort$ is not $\Sp_{2n}$ for any $n \ge 1$, $[\gshort, \gshort]$  is the image of $\Lie(\widetilde{\Gshort})$ in $\gshort$ \cite[Lemma 3.1]{GG:large}, and in particular it is the subalgebra generated by the root subalgebras of $\g$ corresponding to short roots, which is $\gn$.
\end{proof}

\begin{shortroot} \label{gshort.prop}
Let $G$ be a simple and simply connected algebraic group such that $\car k$ is special for $G$, and let $V$ be a representation of $G$.  If $V$ has a subquotient $W$ such that $W^\gn = 0$ and 
\[
\dim W > \begin{cases}
64 & \text{if $G = F_4$} \\
20 & \text{if $G = G_2$} \\
30 & \text{if $G = \Sp_6$} \\
4n^2 & \text{if $G = \Spin_{2n+1}$ with $n \ge 2$ or $G = \Sp_{2n}$ with $n \ge 4$,}
\end{cases}
\]
then $\gn_v \subseteq (\gshort)_v \subseteq \z(\gshort)$ for generic $v \in V$.
\end{shortroot}

\begin{proof}
We verify, for noncentral $x \in \gshort$ such that $x^{[p]} \in \{0, x\}$, that
\begin{equation} \label{gshort.ineq}
\dim x^{\Gshort} + \dim W^x < \dim W.
\end{equation}

Note that $W^{[\gshort, \gshort]} = W^\gn = 0$.  If $G = \Spin_{2n+1}$ for $n \ge 4$, we apply Lemma \ref{heis}\eqref{heis.gen}.  In the other cases, we apply \cite[Th.~12.2]{GG:large}.  These results verify \eqref{gshort.ineq} in all cases.

Then $\dim x^{\Gshort} + \dim V^x < \dim V$, so, for generic $v \in V$, it follows that $(\gshort)_v \subseteq \z(\gshort)$.
\end{proof}

\section{Type $F_4$ or $G_2$} \label{FG.sec}

Suppose $G$ has type $F_4$ or $G_2$ and $\car k = 2$ or 3 respectively.  The maximal ideal $\m$ equals the kernel $\gn$ of the very special isogeny; it is the unique nonzero and proper ideal of $\g$.  Both $\gn$ and $\g/\gn$, as Lie algebras, are the simple quotient $\gt / \z(\gt)$, where $\gt = \spin_8$ or $\sl_3$ respectively.  

The arguments used in the previous two sections can be extended slightly to give a result that will be sufficient to handle the cases where $G = F_4$ or $G_2$.

\begin{prop} \label{FG.dim}
Let $G$ be a simple algebraic group of type $F_4$ or $G_2$ over a field $k$ such that $\car k = 2$ or $3$ respectively. 
Let  $\rho \!: G \to \GL(V)$ be a representation of $G$.
If $V$ has a $G$-subquotient $W$ with $W^\gn = 0$ and $\dim W >  240$ or $48$ respectively, then for generic $v \in V$, $\g_v  =0$.
\end{prop}

\begin{proof}
We will first verify that, for $x \in \g \setminus \gn$ with $x^{[p]} \in \{ 0, x \}$, we have
\begin{equation} \label{cheap.1}
\dim x^G + \dim W^x < \dim W.
\end{equation}

Suppose first that $x \in \g \setminus \gn$ is a long root element, and in particular there is a maximal torus $T$ so that $y$ is a root element in the long root subalgebra $\glong$.  Then $e$ $\Glong$-conjugates of $G$ suffice to generate $\glong$, where $(G, e) = (F_4, 4)$ or $(G_2, 3)$ by \cite[Prop.~10.4, 6.4]{GG:large}. 
As a representation of $\Glong$, $\g$ is a sum of $\glong$ and three inequivalent 8-dimensional representations (for type $F_4$) or two inequivalent 3-dimensional representations (type $G_2$), so $e+1$ $G$-conjugates of $x$ will generate $\g$.  As
\[
(e+1) \dim x^G \le (e+1) (\dim G - \rank G) < \dim W,
\]
\eqref{cheap.1} follows.

If $x \in \g \setminus \gn$ is nilpotent, then as in \cite[Remark 11.3]{GG:large} there is a long root element $y$ in the closure of $x^G$.  By the previous paragraph, $e+1$ $G$-conjugates of $y$ generate $\g$, and as in \cite[Lemma 4.3(1)]{GG:large} the same holds for $x$.
 Again \eqref{cheap.1} follows.

If $x \in \g \setminus \gn$ is toral, then it can be expressed as $\sum c_\alpha h_\alpha$ where the sum ranges over the simple roots $\alpha$ and $h_\alpha \in \Lie(T)$ corresponds to the coroot $\alpha^\vee$.  As $x \not\in \gn$, $c_\beta \ne 0$ for some long simple root $\beta$.  Arguing as in \cite[Example 4.1]{GG:large} we deduce that a long root element $x_\beta$ lies in the closure of $x^{\Gm G}$ and again we have verified \eqref{cheap.1}.

As the nilpotent and toral elements of $\g_v$ lie in $\gn$ for generic $v \in V$, so does all of $\g_v$.  Finally we apply the short root proposition, Prop.~\ref{gshort.prop}, to see that $\g_v = 0$.
\end{proof}

\subsection*{Restricted irreducible representations} Let $G = F_4$ and $\car k = 2$ or $G = G_2$ and $\car k = 3$, and suppose
$\rho \!: G \to \GL(V)$ is an irreducible representation  with restricted highest weight $\la$.    We aim to prove Theorem \ref{MT.restricted} in this case.  

If $\dim V \le \dim G$, then by A.50 and A.49 in \cite{luebeck}, $V$ is either the natural module (of dimension 26 or 7, respectively) or the irreducible quotient $\g/\gn$ of the adjoint representation.  For $\rho$ the natural module, $\ker \drho = 0$ and a generic vector has stabilizer of type $D_4$ or $A_2$ respectively (of dimension 28 or 8 respectively).  Note that this stabilizer has dimension larger than $\dim \g/\n$, so it meets $\n$, the image of $\g$ under the very special isogeny.  It follows that composing the natural representation with the very special isogeny gives a representation with $\ker \drho = \gn$ that is not virtually free; this is $\g/\gn$.

If $\dim V > 240$ or 48 respectively, then $V$ is virtually free by  Proposition \ref{FG.dim}.  Table A.50 in \cite{luebeck} shows that we have considered all restricted irreducible representations of $F_4$, so the proof of Theorem \ref{MT.restricted} is complete in that case.

For $G_2$, there are two remaining possibilities for $\rho$, according to Table A.49.  The first, with highest weight $2\omega_2$ (numbered as in Table \ref{luebeck.table}), has dimension 27 and $\ker \drho = 0$.  It factors through the representation of $\SO_7$ on the irreducible component of $\Sym^2(k^7)$.  As that representation is generically free for $\so_7$ by \cite[Lemma 13.1]{GG:large}, so is $\drho$.  Alternatively, one can verify that this $\drho$ is virtually free using a computer.

The last possibility for $\rho$, with highest weight $2\omega_1$, is obtained by composing the representation in the preceding paragraph with the very special isogeny.  This representation is virtually free by the considerations in the previous paragraph, or by Prop.~\ref{glong.prop}, completing the proof of Theorem \ref{MT.restricted} for $G$ of type $G_2$.


\section{Type $B_n$ with $n \ge 2$} \label{B.sec}

For $G = \Spin_{2n+1}$ for some $n \ge 2$ over a field $k$ of characteristic 2, the Lie algebra $\g$ is uniserial where the short root subalgebra $\n$ is the Heisenberg Lie algebra $\lsub$ from \S\ref{heis.sec}.

Any representation of $G$ is a direct sum $V_1 \oplus V_2$ where $V_1^{\z(\g)} = 0$ and $\z(\g)$ acts trivially on $V_2$; these are just the eigenspaces
of $\z(\g)$. 

\subsection*{Representations with $V^{\z(\g)} = 0$}

\begin{eg}[spin representation] \label{B.spin}
The spin representation $V := L(\omega_1)$  of $G$ (where we number the weights of $G$ as in Table \ref{luebeck.table}) is generically free if and only if $n \ge 7$ \cite{GG:spin}, if and only if $\dim V > \dim G$.  
We remark that one can check with a computer that for $n = 2, 3, 4, 5, 6$, a sum of $4,4,3,2,2$ copies of $V$ is generically free for $\spin_{2n+1}$.
\end{eg}

\begin{eg} \label{B.spin.tensor}
If $V$ is an irreducible representation of $G$ and $V^{\z(\g)} = 0$, then $V \cong L(\omega_1) \ot W$ for some irreducible representation $W$.  This follows from the discussion in section \ref{struct} because $\delta_1$ is the only short simple root.
\end{eg}

\begin{eg} \label{B.nil}
Suppose $x \in \so_{2n+1}$ has $x^{[2]} = 0$ and rank $r > 0$.  The largest possible conjugacy classes for $x$ have a 4-dimensional indecomposable summand $W_2(2)$ or a 3-dimensional indecomposable summand $D(2)$ (following the notation in \cite{Hesselink} or \cite[\S5.6]{LiebeckSeitz}), and the centralizer in $\SO_{2n+1}$ of one of these largest classes has dimension
\[
\sum_{i=1}^r 2(i-1) + \sum_{i=r+1}^{2n+1-r} = \binom{2n+1-r}{2} + \binom{r}{2}.
\]
Consequently, $\dim x^{\SO_{2n+1}} \le r(2n+1-r)$.  (Compare \cite[Example 10.5]{GG:large} for $\SO_{2n}$.)
\end{eg}

\begin{lem} \label{B.faithful}
Let $G = \Spin_{2n+1}$ for some $n \ge 2$ over a field $k$ of characteristic 2, and suppose that $V$ is a representation of $G$ such that $V^{\z(\g)} = 0$.  Then for noncentral $x \in \g$, $\dim V^x \le \frac34 \dim V$.  Moreover, if (1) $\dim V > 4n^2 + 4n$ or (2) $n \ge 8$, then  (a) $\dim x^G + \dim V^x < \dim V$ for all noncentral $x \in \g$ such that $x^{[2]} \in \{ 0, x \}$ and (b) $V$ is generically free for $\g$.
\end{lem}

\begin{proof}
We first prove that $\dim V^x \le \frac34 \dim V$.  By passing to orbit closures, it suffices to prove this in case $x$ is nilpotent.  The crux case is where $V$ is the spin representation, where the claim holds if $n = 2$ (because $\dim V = 4$) and if $n \ge 3$ by \cite[Prop.~2.1(i)]{GG:spin}.  
  
 If $V$ is irreducible, then it is $L(\omega_1) \ot W$ for some irreducible $W$.
 In this case, $\dim V^x \le (\dim L(\omega_1)^x) (\dim W)$  as in \cite[Lemma 11.1]{GG:irred}, proving the claim. 
Finally, for a composition series $0 = V_0 \subset V_1 \subset \cdots \subset V_r = V$ of $V$, we have $(V_i/V_{i-1})^{\z(\g)} = 0$ because $\z(\g)$ acts semisimply on $V$ and $\dim (V')^x \le \frac34 \dim V$ for $V' := \oplus V_i/V_{i-1}$, so also for $V$, proving that $\dim V^x \le \frac34 \dim V$.

Now fix noncentral $x \in \g$ such that $x^{[2]} \in \{ 0, x \}$.  If $x^{[2]} = 0$, then the image $\bx \in \so_{2n+1}$ of $x$ --- as a $(2n+1)$-by-$(2n+1)$ matrix --- has rank $r > 0$ so $\dim \bx^{\SO_{2n+1}} \le r(2n-r+1)$, whence $\dim x^{\Spin_{2n+1}} \le r(2n-r+1)$ because the map $\spin_{2n+1} \to \so_{2n+1}$ is injective on nilpotents.  (Indeed, if $x$ and $x+z$ are square-zero and $z$ is central in $\spin_{2n+1}$, then $0 = (x+z)^{[2]} = z^{[2]}$, so $z = 0$.)
If $x^{[2]} = x$, then the centralizer of $x$ has type $D_r \times B_{n-r}$ for some $r$ and we find the same formula for $\dim x^G$.  The dimension is maximized for $r = (2n+1)/4$, so $\dim x^G \le n^2 + n$.  Thus, (1) implies (a).

Suppose (2), $n \ge 8$.  If $V$ is the spin representation, then $\dim x^G + \dim V^x < \dim V$ as in \cite{GG:spin}.  If $V$ is restricted irreducible but not the spin representation, i.e., $V \cong (\mathrm{spin}) \ot L(\la_\ell)$, then $\dim V \ge 2n 2^n > 4n^2 + 4n$ and again (a) holds.  Claim (a) follows for general $V$ as in the second paragraph of this proof.

(a) implies (b) as recalled in \S\ref{back.sec}.
\end{proof}

Note the following corollary.

\begin{cor}  \label{B.irr1}  Let $G = \Spin_{2n+1}$ for some $n \ge 2$ over a field $k$ of characteristic 2, and suppose that $V$ is an
irreducible representation of $G$ such that $V^{\z(\g)} = 0$.    If $\dim V > \dim G$, then
$V$ is generically free.   
 \end{cor}

\begin{proof} Write $V=L(\omega_1) \otimes W$ as in Example \ref{B.spin.tensor}.  If $W$ is trivial, then the claim is from Example \ref{B.spin}.  If $n \ge 8$, then the claim is Lemma \ref{B.faithful}.  So suppose $W$ is nontrivial and $2 \le n < 8$.  As $\dim W \ge 2n$, we have
$\dim V > 4n^2 + 4n$ unless $n =2$ or 3 and $\dim W = 2n$.

So suppose $n = 2$ or 3 and $W$ is the orthogonal module or a nontrivial Frobenius twist of it.  In the latter case, $V$ is a direct sum of $2n$ copies of the spin module and so is generically free (Example \ref{B.spin}).  In the former case, one can verify with a computer that $V$ is generically free, as was recorded in \cite[Examples 5.2, 5.3]{GG:irred}.
\end{proof} 

Note that the only irreducible modules with $V^{\z(\g)} = 0$ and $\dim V \le \dim G$ are the spin modules for $n \le 6$.  

\subsection*{Representations killed by $\z(\g)$}
We have dealt with those representations $V$ of $G$ such that $V^{\z(\g)} = 0$.  If $\z(\g)V = 0$, then the highest weight of each of the composition factors of $V$ lies in the root lattice.  We have:

\begin{lem}  \label{B.roots}
In a root system of type $B_n$ ($n \ge 2$), for $\la$ in the root lattice and $\alpha$ a short root, $\qform{\alpha^\vee, \la}$ is an even integer.
\end{lem}

\begin{proof}
Because the Weyl groups acts transitively on short roots, we may assume that $\alpha^\vee$ is the short simple root.  Because the expression $\qform{\alpha^\vee, \la}$ is linear in $\la$, we may assume that $\la$ is a simple root.  Then the claim follows from looking at the Cartan matrix.
\end{proof}

We note for later use:

\begin{lem} \label{SO.faithful}
If $\car k = 2$ and $\rho \!: \SO_{2n+1} \to \GL(V)$ is an irreducible representation for some $n \ge 2$, then  $\rho$ is not faithful and the composition of $\spin_{2n+1} \to \so_{2n+1}$ with $\drho$ vanishes on $\n$.
\end{lem}

\begin{proof}
The highest weight $\la$ of $\rho$ is in the root lattice (because $\SO_{2n+1}$ is adjoint), so Lemma \ref{B.roots} shows that $\la$ vanishes on coroots corresponding to short roots.  Thus by \cite[Th.~11.1]{St:rep} the composition $\spin_{2n+1} \to \so_{2n+1} \to \gl(V)$ vanishes on the ideal $\gn$ of $\spin_{2n+1}$ for $\gn$ as in \S\ref{struct}, which has nonzero image in $\so_{2n+1}$.
\end{proof}

\begin{rmk}
Lemma \ref{B.roots} can be viewed, by way of the duality between the root systems of types $B$ and $C$, as equivalent to the statement that for type $C$ every long root is 2 times a weight.  From this perspective, Lemma \ref{SO.faithful} is the analogue for type $B$ of the well-known fact that, when $\car k = 2$, Cartan subalgebras in $\sp_{2n}$ are properly larger than maximal toral subalgebras.  

By the way, Lemmas \ref{B.roots} and \ref{SO.faithful} also apply to type $A_1$, mutatis mutandis.  See \cite[Example 3.3]{GG:large} for the version of Lemma \ref{SO.faithful}.
\end{rmk}

The very special isogeny $G = \Spin_{2n+1} \to \Sp_{2n}$ is another way of viewing the trivial statement that the alternating bilinear form on the natural module of $G$ is $G$-invariant.  It factors through $\Spin_{2n+1} \to \SO_{2n+1}$, and the image $\g/\n$ of $\g$ in $\sp_{2n}$ is  isomorphic to the derived subalgebra of $\so_{2n}$, which is a simple $G$-module (i.e., $\gn = \m$) if $n$ is odd and has a 1-dimensional center if $n$ is even (i.e., $\gn$ has codimension 1 in $\m$).  


\begin{lem}  \label{B.vs}
Let $\rho \!: \Spin_{2n+1} \to \GL(V)$ be a representation over a field of characteristic $2$, for some $n \ge 3$.  If $\rho$ factors through the very special isogeny, $V^{\spin_{2n+1}} = 0$, and $\dim V > 4n^2$, then $\spin_{2n+1}$ and $\so_{2n+1}$ act virtually freely on $V$.
\end{lem}

\begin{proof}
By hypothesis, $\rho$ factors through $\Spin_{2n+1} \to \SO_{2n+1} \to \Sp_{2n}$.  The image of $\so_{2n+1}$ in $\sp_{2n}$ is $\so_{2n}$, the unique maximal $\Sp_{2n}$-invariant ideal in $\sp_{2n}$ and the Lie algebra of a subgroup $\SO_{2n} \subset \Sp_{2n}$.  It suffices to verify that $\so_{2n}$ acts virtually freely on $V$.

The image of $\spin_{2n+1}$ in $\so_{2n}$ is $[\so_{2n}, \so_{2n}]$, so $V^{[\so_{2n}, \so_{2n}]} = 0$.  Applying now \cite[Th.~A]{GG:large} gives the claim.
\end{proof}

\begin{cor} \label{B.quo} 
Let $V$ be a representation of $G := \Spin_{2n+1}$ for some $n \ge 2$ and assume $\car k = 2$.
If $V^{\g} = 0$ and $\dim V > 8n^2 + 4n$,  then a generic $v \in V$ has $\g_v \subseteq \m$ and $\g$ acts virtually freely on $V$.
\end{cor}

\begin{proof}  Write $V=V_1 \oplus V_2$ as at the start of this section.   If $\dim V_1 > 4n^2+4n$, then $V_1$
and so $V$ is generically free.   Otherwise $\dim V_2 > 4n^2$ and the group of type $D$ has generic
stabilizer contained in its center by \cite{GG:large}.  
\end{proof}

One can do better by intersecting the generic stabilizers for $V_1$ and $V_2$.

\begin{eg}[natural representation]  \label{B.nat}
Here we treat the natural module, $V:=L(\omega_n)$ for $n \ge 2$.  We have $\dim V = 2n < \dim G$.  As in the proof of Lemma \ref{B.vs}, the image of $\so_{2n+1}$ in $\gl(V)$ is a copy of $\so_{2n}$ which acts on $V$ with generic stabilizer $\so_{2n-1}$, hence $\so_{2n+1}$ acts on $V$ with kernel $\gn$ and generic stabilizer $\gn + \so_{2n-1}$.

Note that $V^\g = V^{\so_{2n}} = 0$.
For $W := \oplus^c V$ with $c > 2n$,  Lemma \ref{B.vs} says that $W$ is virtually free.  (Compare \cite[Example 10.3]{GG:irred} for the case $\car k \ne 2$.)
\end{eg}

\begin{eg}[adjoint representation]  \label{B.adj}
Here we treat $V:=L(\omega_{n-1})$ for $n \ge 3$, the irreducible quotient of the Weyl module $\spin_{2n+1}$.   As in \S\ref{long.sec}, the long root subalgebra $\glong$ is $\spin_{2n}$, and $\spin_{2n+1} = \n + \spin_{2n}$. 
This $V$ 
 is the irreducible quotient of the $\spin_{2n}$-module $\spin_{2n}$.  By \cite[Example 3.4]{GG:large}, the stabilizer in $\spin_{2n}$ of a generic vector $v \in V$
is $\Lie(T)$ for $T$ a maximal torus depending on $v$ and we conclude that $\g_v = \n + \Lie(T)$.  (Alternatively, this representation factors through the very special isogeny and one can find $\g_v$ by pulling back the stabilizer in $\sp_{2n}$ described in Example \ref{C.wedge2}.)
\end{eg}

\subsection*{Restricted irreducible representations, part I}
Let $V$ be a restricted irreducible representation of a group $G$ of type $B_n$ for some $n \ge 2$ over a field $k$ of characteristic 2; we aim to prove Theorem \ref{MT.restricted} for this $G$.  The highest weight $\la = \sum_{i=1}^n c_i \omega_i$ of $V$ has $c_i \in \{ 0, 1 \}$ for all $i$.  If $\la = 0$, equivalently $\ker \drho = \g$, then there is nothing to do. 

If $c_1 \ne 0$, then we are in the case of Example \ref{B.spin.tensor} and Corollary \ref{B.irr1}, so $\ker \rho = 0$ and there is nothing more to do.

\subsection*{Restricted irreducible representations, part II}
We continue now the proof of Theorem \ref{MT.restricted} for type $B$, except now we assume that $c_1 = 0$.  Thus, the highest weight $\la$ belongs to the root lattice and the representation $\rho$ factors through not just $\SO_{2n+1}$ but $\Sp_{2n}$.  To prove Theorem \ref{MT.restricted}, it suffices to show that (1) $\spin_{2n+1}$ does not act virtually freely when $\dim V \le \dim G$ and (2) $\so_{2n+1}$ does act virtually freely when $\dim V > \dim G$.

Assume first that $n > 11$.  
Applying Lemma \ref{B.vs}, we may assume that $\dim V \le 4n^2$, so $\dim V < n^3$. By \cite{luebeck}, $V$ is either the $2n$-dimensional representation $L(\omega_n)$ as in Example \ref{B.nat} or it is $L(\omega_{n-1})$ as in Example \ref{B.adj}.

Next consider $4  \le n \le 11$.  Applying again Lemma \ref{B.vs}, we may assume that $\dim V \le 4n^2$.  Examining the tables in \cite{luebeck}, we find that the only possibilities for $V$ that have not yet been considered are the cases where $n = 4$ or 5, $\la = \omega_{n-2}$, and $\dim V = 48$ or 100 respectively.
 In both cases, $\gn$ is in the kernel of the representation and computer calculations with Magma as in \cite{GG:irred} produces a vector in $V$ with stabilizer $\gn$.  Note that $\dim V > \dim G$ in both cases.
 
 If $n=2$, the only restricted irreducible module is the orthogonal $4$-dimensional module and 
 clearly the result holds. 
 
 If $n=3$,  the only irreducible restricted modules with $c_1=0$
 are the orthogonal module (Example \ref{B.nat}), the module with high weight $\omega_2$ of dimension $14$ (Example \ref{B.adj}), and the module with high weight $\omega_2 + \omega_3$ of dimension $64$.  We verify with a computer that Theorem \ref{MT.restricted} holds in
 the latter case.

\begin{eg}[$\Spin_9$] \label{B3}
For later reference, we examine more carefully the spin representation $V$ of $G = \Spin_9$ when $\car k = 2$.  The stabilizer $G_v$ of a generic vector $v \in V$ is isomorphic to $\Spin_7$ \cite{GG:spin} and is contained in a long root subgroup $\Spin_8$ in $G$ in such a way that the composition $\Spin_7 \subset \Spin_8 \subset \Spin_9 \to \SO_9$ is injective.  (See for example \cite{Raja} for a discussion of the first inclusion in the case $k = \mathbb{R}$.)  In particular, $\z(\spin_7) \ne \z(\spin_9)$.

Now $\z(\spin_9) \subset \z(\spin_8)$, where the terms have dimensions 1 and 2 respectively.   We claim that furthermore $\z(\spin_7) \subset \z(\spin_8)$ so $\z(\spin_7) + \z(\spin_9) = \z(\spin_8)$.  To see this, restrict $V$ to $\Spin_8$ to find a direct sum $V_1 \oplus V_2$ of inequivalent 8-dimensional irreducible representations.  One $V_i$ restricts to be the spin representation of $\Spin_7$ and the other is uniserial with composition factors of dimension 1, 6, 1, corresponding to the inclusion $\SO_7 \subset \SO_8$.  From this we can read off the action of the central $\mu_2$ of $\Spin_7$ on $V$ and we find that it is central in $\Spin_8$, proving the claim.

We have $\spin_8 \cap \m = \z(\spin_8)$ by Lemma \ref{glong} and $\spin_8 \cap \n = \z(\spin_9)$.  Using that $\g_v = \spin_7$, we conclude that $\n_v = 0$ and $\m_v = \z(\spin_7)$.
\end{eg}


\section{Type $C_n$ with $n \ge 3$} \label{C.sec}

\subsection*{Alternating bilinear forms} 
The simply connected group $\Sp_{2n}$ of type $C_n$ can be viewed as the automorphism group of a nondegenerate alternating bilinear form.  We now recall some facts about this correspondence.

\begin{eg} \label{adj.eg}
Suppose $b$ is a nondegenerate alternating bilinear form on a finite-dimensional vector space $V$ over a field $F$ (of any characteristic).  This gives an ``adjoint'' involution $\sigma \!: \End_F(V) \to \End_F(V)$ such that $b(Tv,v') = b(v,\sigma(T)v')$ for all $T \in \End_F(V)$ and $v,v' \in V$.  If $x \in \End_F(V)$ is such that $\sigma(x) = \pm x$, then we find an equation $b(xv,v') = \pm b(v, xv')$.  By taking $v \in \ker x$ or $(\im x)^\perp$  and allowing $v'$ to vary over $V$ we find each of the containments between $\ker x$ and $(\im x)^\perp$, i.e., $\ker x = (\im x)^\perp$.  If additionally $x^2 = x$ (i.e., $x$ is a projection), then $V$ is an orthogonal direct sum $(\ker x) \oplus (\im x)$.

If $x \in \End_F(V)$ is such that $\sigma(x)x = 0$, then $b(xv, xv') = 0$ for all $v, v' \in V$ and we find that $\im x \subseteq (\im x)^\perp$, i.e., $\im x$ is totally singular.
\end{eg}

We may view $\Sp_{2n}$ as the subgroup of $\GL_{2n}$ preserving the bilinear form $b(v,v') := v^\top J v'$ where $J = \left( \begin{smallmatrix} 0 & I_n \\ -I_n & 0 \end{smallmatrix} \right)$.  The Lie algebra $\sp_{2n}$ of $\Sp_{2n}$ consists of matrices $\left( \begin{smallmatrix} A & B \\ C & -A^\top \end{smallmatrix} \right)$ for $A, B, C \in \gl_n$ such that $B^\top = B$ and $C^\top = C$.  In the notation of Example \ref{adj.eg}, $\sigma(g) = -Jg^\top J$ and $\sp_{2n}$ consists of those $x \in \gl_{2n}$ such that $\sigma(x) + x = 0$.  (In particular, $\ker x = (\im x)^\perp$ for $x \in \sp_{2n}$.   If moreover $\car k = 2$ and $x^{[2]} = 0$, then $\sigma(x) x  = x^{[2]} = 0$, so $\im x$ is totally singular.)

The group $\GSp_{2n}$ of similarities of $b$ is the sub-group-scheme of $\GL_{2n}$ generated by $\Sp_{2n}$ and the scalar transformations.  Its Lie algebra consists of those $x \in \gl_{2n}$ such that $\sigma(x) + x \in kI_{2n}$, i.e., $x$ of the form $\left( \begin{smallmatrix} A & B \\ C & \mu I_n -A^\top \end{smallmatrix} \right)$ for $A, B, C \in \gl_n$ and $\mu \in k$, where $B^\top = B$ and $C^\top = C$.  Then $\sp_{2n} \subset \gsp_{2n}$, the quotient $\GSp_{2n} / \Gm \cong \Sp_{2n} / \mu_2$ is the adjoint group $\PSp_{2n}$, and the natural map $\gsp_{2n} \to \psp_{2n}$ is surjective.

The preceding two paragraphs apply to any field $k$; we now explicitly assume $\car k = 2$ and describe the toral elements in $\gsp_{2n}$, i.e., those $x$ such that $x^2 = x$.  As such an $x$ is a projection, it gives a decomposition of $k^{2n}$ as a direct sum of vector spaces $(\ker x) \oplus (\im x)$.  If $x$ belongs to $\sp_{2n}$, then this is an orthogonal decomposition as in Example \ref{adj.eg} where the restrictions of $b$ to $\ker x$ and $\im x$ are nondegenerate.  Otherwise, $\sigma(x) + x = I_{2n}$, so
\[
b(xv, xv') = b((I_{2n} - \sigma(x))v, xv') = b(v, xv') - b(v, x^2 v') = 0  \quad (v, v' \in V).
\]
That is, $\im x$ is totally isotropic.  Analogously, $1 - x \in \gsp_{2n}$ is toral and so $\ker x = \im(1-x)$ is also totally isotropic.  In sum, we find that $\ker x$ and $\im x$ are maximal totally isotropic subspaces.

\subsection*{Dimension bounds}
For the remainder of this section, we set $G = \Sp_{2n}$ with $n \ge 3$ over a field $k$ of characteristic 2. 
We first make some remarks about nilpotent elements of square $0$.

\begin{lem} \label{spsq}
Let $x \in \sp_{2n}$ for $n \ge 3$ have $x^{[2]} = 0$ and rank $r > 0$.  Then
 $\dim x^{\Sp_{2n}} \le r(2n+1)-r^2$.
\end{lem}

\begin{proof}[Proof \#1] 
Fix a totally singular $r$-dimensional subspace $W$ of the natural module $V$ and let $P$ be the
maximal parabolic subgroup of $\Sp_{2n}$ stabilizing $W$.  
Let $C$ be the set of elements in $\g$ with $y^{[2]}=0$ and $\im y \subseteq W$.
As above,  $\ker y =(\im y)^{\perp}$ 
and in particular $\ker y \supseteq W^{\perp}$.   
So we see that $C$ is the center of the nilpotent radical of $\Lie(P)$ and can
be identified with the space of $r$-by-$r$ symmetric matrices.  In particular,
$\dim C = r(r+1)/2$.

Now let $x \in \g$ with $x^{[2]}=0$ and $\im x =W$.  Consider the map
$f: G \times C \rightarrow \g$ given by $f(g,y)=\Ad(g)y$.   Every fiber has
dimension at least $\dim P$, since $f(P \times C) = C$.   Thus the dimension
of the image of $f$ is at most $\dim C + \dim G/P = r(r+1) + 2r(n-r) = r(2n+1) - r^2$.
Since $x^G$ is contained in the image of $f$, we obtain the same inequality 
for $\dim x^G$.

We remark that, by Richardson, the Levi subgroup of $P$ has a dense orbit on $C$, whence the largest
such class  has dimension  precisely  $r(2n+1) - r^2$.   Note also that the fact that $\im x$ is totally
singular shows that $x$ is a sum
 of commuting root elements.   If $r$ is odd, then we can always
take $x$ to be a sum of commuting long root elements.   If $r$ is even, then we can always take
$x$ to be a either a commuting sum of long root elements (the larger class) or a sum of commuting
short root elements.   Note that the smaller class is contained in the closure of the larger class. 
It is well known (cf.~\cite[Chap.~4]{LiebeckSeitz}) that if $x \in \g$ has even rank   and $x^{[2]}=0$, then
$x$ is an element of $\so_{2n}$.  
\end{proof}

\begin{proof}[Proof \#2]
There are two possibilities for the conjugacy class of $x$ in case $r$ is even, see \cite[4.4]{Hesselink} or \cite[p.~70]{LiebeckSeitz}.  We focus on the larger class; regardless of the parity of $r$ we may assume that the restriction of the natural module to $x$ includes a 2-dimensional summand denoted by $V(2)$ in \cite{LiebeckSeitz}.  For this $x$, the function denoted by $\chi$ in the references amounts to $1\mapsto0$ and $2 \mapsto 1$.  The formulas in these references now give that the centralizer of $x$ in $\Sp_{2n}$ has dimension
\[
\sum_{i=1}^r (2i-1) + \sum_{i=r+1}^{2n-r} i
 = 2n-r + \binom{2n-r}{2} + \binom{r}{2}.\qedhere
\]
\end{proof}

The Lie algebra $\g = \sp_{2n}$ has derived subalgebra $\m = \so_{2n}$ of codimension $2n$; it is the unique maximal $G$-invariant ideal in $\g$.  The short root subalgebra $\n = \ker \dpi$ has codimension 1 in $\m$ and is $[\m, \m]$.  The quotient $\g/\n$ is the Heisenberg Lie algebra from section \ref{heis.sec}.  The unique maximal ideal $\m$ also contains $\z(\g)$ (dimension 1).  The center $\z(\g)$ is  contained in $\n$ if and only if $n$ is even.

\begin{lem} \label{C.gen}
Let $G = \Sp_{2n}$ for some $n \ge 3$ over a field $k$ of characteristic $2$, and let $x \in \sp_{2n}$ be noncentral.
\begin{enumerate}
\item \label{C.gen.tor} If $x$ is toral, 
then 
 $\max \{ 4, \lceil {n/s} \rceil \}$ conjugates of $x$ suffice to generate a subalgebra containing $\n$, where $2s$ is the dimension of the smallest eigenspace.
 
\item \label{C.gen.even} If $x^{[2]}=0$
and $x$ has even rank $2s > 0$ or odd rank $2s + 1 \ge 3$, then 
 $\max \{ 4, \lceil{n/s} \rceil \}$ conjugates of $x$ generate a subalgebra containing $\n$.

\item \label{C.gen.1} If $x^{[2]} = 0$ and $x$ has rank $1$, then $\max\{ 8, 2n \}$ conjugates of $x$ generate a subalgebra containing $\gn$.
\end{enumerate}
\end{lem}

\begin{proof}  If $x^{[2]}$ is toral, or is nilpotent with even rank, then $x$ is in $\m = \so_{2n}$ (cf.~the remarks in the proof of Lemma \ref{spsq}), and the claim is \cite[Prop.~10.4]{GG:large}.

If $x^{[2]} = 0$ and $x$ rank $2s + 1 \ge 3$, the description of the classes given in the proof of Lemma \ref{spsq} shows that the closure of $x^{\Sp_{2n}}$ contains a $y$ of rank $2s$ such that $y^{[2]} = 0$.  As $\max \{ 4, \lceil n/s \rceil \}$ conjugates of $y$ suffice to generate a subalgebra containing $\gn$, the same holds for $x$. 

If $x^{[2]} = 0$ and $x$ has rank 1, we choose $y$ conjugate to $x$ such that $(x+y)^{[2]}=0$ and $x+y$ has rank $2$.  
By \eqref{C.gen.even},  $\max \{ 4, n \}$ conjugates of $x + y$  generate a subalgebra containing $\n$, whence \eqref{C.gen.1} holds.
\end{proof}

\begin{prop}  \label{C.faithful}
Let $(V,\rho)$ be a representation of $G = \Sp_{2n}$ or $\PSp_{2n}$ with $n \ge 3$ over a field $k$ of characteristic $2$.
If $V^\gn = 0$ and 
\[
\dim V > 
\begin{cases}
48 & \text{if $G = \Sp_6$ or $\PSp_6$ (i.e., $n = 3$)} \\
72 & \text{if $G = \Sp_8$;} \\
80 & \text{if $G = \PSp_8$; and} \\
6n^2 + 6n&\text{if $n \ge 5$,}
\end{cases}
\]
then  $\g$ acts virtually freely on $V$.
\end{prop}
 
In the statement, in case $G = \PSp_{2n}$, $\rho$ induces a representation of $\Sp_{2n}$ and so it still makes sense to speak of the action of $\gn \subset \sp_{2n}$ on $V$, whence $\g$ acts virtually freely on $V$.

\begin{proof}  
For noncentral $x \in \g$ such that $x^{[2]} \in \{ 0, x\}$, we check  that $x$ does not meet $\g_v$ for generic $v \in V$, and therefore that $\g_v \subseteq \z(\g)$.

\subsubsection*{Case $G = \Sp_{2n}$ for $n \ge 4$}
Suppose first that $x^{[2]} = x$, or $x^{[2]} = 0$ and $x$ has even rank. Then $x$ belongs to $\m = \so_{2n}$ and is not itself central in $\m$.  The Short Root Proposition \ref{gshort.prop} gives that for $\dim V > 4n^2$, $x$ does not meet $\m_v$ for generic $v \in V$.

We may assume that $x^{[2]} = 0$ and $x$ has odd rank $r = 2s+1$, where $\dim x^G$ is bounded as in Lemma \ref{spsq}.   We find an   $e > 0$ such that $e$ conjugates of $x$ generate a subalgebra of $\g$ containing the image of $\n$ and such that $e \cdot \dim x^G \le b(G)$, where $b(G)$ is the right side of the inequality in the statement.  This verifies by Lemmas \ref{ineq} and \ref{ineq.spin}\eqref{mother.lie}  that $x$ does not meet $\g_v$ for generic $v \in V$.
 
If $r = 1$, then applying Lemma \ref{C.gen}\eqref{C.gen.1} gives that $e \cdot \dim x^G \le 4n \max \{ 4, n \}$.

If $r = 3$, then $\dim x^G \le 6(n-1)$.  Applying Lemma \ref{C.gen}\eqref{C.gen.even} gives $e \dim x^G \le 6(n-1) \max\{4,n\}$.

Suppose $r \ge 5$ (so also $n \ge 5$). If $s \ge n/4$, then $4$ conjugates suffice to generate a sublagebra containing $\n$.
As $\dim x^G \le 2r(n-s)$, it suffices to bound $8r(n-s)$.  This is maximized at $r = n + \frac12$ and so
we obtain a bound of $4n^2+4n$.   

If $s < n/4$, then it suffices to bound $(n/s+1)(2r)(n-s)  = (n^2-s^2)(2r/s) \le 5n^2 - 5s^2 
\le 5n^2-20$.   

\subsubsection*{Case $G = \Sp_6$} 
As in the $n \ge 4$ case, we may reduce to the case where $x^{[2]} = 0$ and $x$ has odd rank.

If $x$ has rank $3$,
then $\dim x^G \le 12$ and we get a bound of $48$.  If $x$ is a long root element, then $6$ conjugates suffice to 
generate $\sp_6$ (by reducing to the rank $2$ case and using generation by $3$ conjugates).   Then $6 \dim x^G = 36$
and the result holds. 

\subsubsection*{Case $G = \PSp_{2n}$ for $n \ge 3$}
For $\GSp_{2n} := (\Sp_{2n} \times \Gm)/\mu_2$, the group of similarities of the bilinear form, we have a natural surjection $\GSp_{2n} \to \PSp_{2n}$ whose differential $\gsp_{2n} \to \psp_{2n}$ is surjective and has central kernel $k$, the scalar matrices.  Let $x \in \gsp_{2n}$ be noncentral such that $x^{[2]} \in \{ 0, x \}$.  If $x$
belongs to $\sp_{2n}$, we have already verified that $x$ cannot lie in $(\sp_{2n})_v$ for generic $v \in V$.  So assume that $x$ is toral and does not belong to $\sp_{2n}$, i.e., is the projection on a maximal totally isotropic subspace as in the paragraph preceding the statement.  Up to conjugacy we may assume that $x$ is $\left( \begin{smallmatrix} 0_n & 0_n \\ 0_n & I_n \end{smallmatrix} \right)$.  The nilpotent linear transformation $y := \left( \begin{smallmatrix} 0_n & I_n \\ 0_n & 0_n \end{smallmatrix} \right)$ belongs to $\sp_{2n}$, has rank $n$ and satisfies $y^{[2]}=0$.   Note that $x + ty$ is conjugate to $x$ for any scalar $t$. 
It follows that $y$ is in the closure of $x^{\Gm G}$.   

Lemma \ref{C.gen} gives that 4 conjugates of $y$ suffice to generate a subalgebra of $\sp_{2n}$ containing $\gn$.  Take $M = \n = [\so_{2n}, \so_{2n}]$ and $N = \z(\n)$, so $M/N$ is the irreducible representation $L(\omega_{n-1})$ of $G$ whose highest weight is the highest short root.  As 
\[
\dim \gsp_{2n}/M = 2n+1 < 2n^2-n-1 \le \dim M/N,
\]
\cite[Lemma 4.3(3)]{GG:large} says that 4 conjugates of $x$ also suffice to generate a subalgebra of $\sp_{2n}$ containing $\gn$.  On the other hand, $\dim x^{\GSp_{2n}} = n^2 + n$, giving $e \cdot \dim x^{\GSp_{2n}} \le 4n^2 + 4n$.  Therefore, it suffices to take $b(\PSp_{2n}) = \max \{ b(\Sp_{2n}), 4n^2 + 4n \}$.
\end{proof}

\subsection*{Restricted irreducible representations}
Let $V$ be a restricted irreducible representation of an algebraic group $G$ of type $C_n$ with $n \ge 3$ over $k$ of characteristic 2; we aim to prove Theorem \ref{MT.restricted} for this $G$.  The highest weight $\la = \sum_{i=1}^n c_i \omega_i$, numbered as in Table \ref{luebeck.table}, has $c_i \in \{ 0, 1 \}$ for all $i$.  If $\la = 0$, equivalently $\ker \drho = \g$, then there is nothing to do.

\begin{eg}[``spin'' representation]  \label{C.spin}
Put $G' := \Spin_{2n+1}$.  Composing the very special isogeny $\Sp_{2n} \to G'$ with the (injective) spin representation of $G'$ of dimension $2^n$ yields the irreducible representation $(V,\rho) = L(\omega_1)$ of $\Sp_{2n}$.  For $n \ge 7$, $\dim V > \dim G$ and the stabilizer $(\g')_v$ of a generic $v \in V$ is zero \cite[Th.~1.1]{GG:spin}, so $(\sp_{2n})_v = \ker \drho$.

For $3 \le n < 7$, $\dim V < \dim \Sp_{2n}$ and one can check using a computer that there exist vectors in $V$ whose stabilizer is $\ker \drho$, therefore, $\sp_{2n}$ acts virtually freely on $V$.  Alternatively, for $n = 3, 5, 6$ and generic $v \in V$, $(\Spin_{2n+1})_v$ is $G_2$, $\SL_5 \rtimes \Z/2$, $(\SL_3 \times \SL_3) \rtimes \Z/2$ respectively \cite{GG:spin}; each of these has a Lie algebra that is a direct sum of simples, and so it can not meet the solvable ideal that is the image of $\sp_{2n}$ in $\spin_{2n+1}$.  For these $n$, $\dim V < \dim G$.

For $n$ even, the representation factors through $\PSp_{2n}$.  For $n \ge 8$ and $n = 6$, $\psp_{2n}$ acts virtually freely as in the preceding two paragraphs. For $n = 4$, the image of $\psp_8$ in $\spin_9$ is the maximal proper $\Spin_9$-submodule, for which a generic vector $v$ in the spin representation has a 1-dimensional stabilizer (Example \ref{B3}), i.e., $\dim (\psp_8)_v / \ker \drho = 1$, i.e., $\psp_8$ does not act virtually freely.
\end{eg}

Therefore, we have addressed the cases where $\la_s = 0$, i.e., $\la \in \{ 0, \omega_1 \}$, so we now assume that $\la_s \ne 0$, whence, $V^{\gn} = 0$.  As this excludes all the representations in Table \ref{irred.vfree},
our task is to prove that (1) $\sp_{2n}$ does not act virtually freely if $\dim V \le 2n^2 + n$ and (2) $\psp_{2n}$ acts virtually freely if $\dim V > 2n^2 + n$.

\begin{eg}[natural representation]  \label{C.nat}
Here we treat $V:=L(\omega_n)$, the natural representation of $\Sp_{2n}$, which has $\dim V = 2n < \dim \Sp_{2n}$.  In this case $\Sp_{2n}$ is transitive on all nonzero vectors and so we see
the generic stabilizer is the derived subalgebra of the maximal parabolic subalgebra that is the stabilizer
of a $1$-dimensional space. 
\end{eg}

\begin{eg}[``$\wedge^2$'' representation]  \label{C.wedge2} 
Let $V:=L(\omega_{n-1})$, which has $\dim V < \dim G$.  We will show that the generic stabilizer
in $\Sp_{2n}$ is $A_1 \times \cdots \times A_1$ (more precisely the stabilizer of
$n$ pairwise orthogonal two-dimensional non-degenerate subspaces), so $V$ is not virtually free for $\sp_{2n}$ nor for $\psp_{2n}$.

Let $M$ be the natural module for $\Sp_{2n}$, of dimension $2n$.
Let $W = \wedge^2 M$ which we can identify with the set of
skew adjoint operators on $M$, i.e.,  the linear maps $T\!:M \to M$
with $(Tv,v)=0$ for all $v \in M$.    If $n$ is odd, then 
$W \cong k \oplus V$.  If $n$ is even, then $W$ is uniserial with $1$-dimensional
socle and head with $V$ as the nontrivial composition factor. 
The unique submodule  $W_0$ of codimension $1$ in each case consists of those
elements $T$ with reduced trace equal to $0$.  

It follows as in \cite{GoGu} or \cite[Example 8.5]{GG:simple} that a generic element $T$ of $W$ 
(or $W_0$) is semisimple
and has $n$ distinct eigenvalues each of multiplicity $2$.   It follows
that, as a group scheme, the stabilizer of such an element is 
as given and so the result holds for $n$ odd.

Assume that $n \ge 4$ is even.  Note that a generic element $T$
of $W$ also has the property that the $n(n-1)/2$ differences of
the eigenvalues of $T$ are distinct.   It follows easily that the stabilizer
of such an element in $V=W_0/W^{\Sp_{2n}}$ is the same as in $W_0$
and the result follows.
\end{eg}

If $n > 11$ and $\dim V \le 6n^2 + 6n$, then $V$ is $L(\omega_n)$ or $L(\omega_{n-1})$ by \cite{luebeck}, so we may assume $3 \le n \le 11$.  Therefore we are reduced to considering Tables A.32--A.40 in \cite{luebeck}, one table for each value of $n$.  In those tables, the representations $V$ not already handled by Prop.~\ref{C.faithful} or Examples \ref{C.spin}, \ref{C.nat}, \ref{C.wedge2}  are: $L(\omega_{n-2})$ (``$\wedge^3$'') for $\Sp_{2n}$ with $n = 4, 5$ and the representation $L(\omega_1 + \omega_3)$ of $\PSp_6$ of dimension 48.  A computer check verifies that $\g$ acts virtually freely on these representations.

As we have verified Theorem \ref{MT.restricted} for groups of type $C$ (in this section), $B$ (in section \ref{B.sec}), and $F_4$ and $G_2$ (in section \ref{FG.sec}), the proof is complete.$\hfill\qed$

\section{A large representation that is not virtually free} \label{eg.sec}

In this section, $G$ is a simple and simply connected algebraic group over $k$ and $\car k$ is special for $G$. The main theorem of \cite{GG:large} includes a statement of the following type: \emph{If $\car k$ is not special and $V$ is a $G$-module with a subquotient $X$ such that $X^{[\g, \g]} = 0$ and $\dim X$ is big enough, then $\g$ acts virtually freely on $V$.}   Example \ref{large.eg} below shows that such a result does not hold when $\car k$ is special.

 As in \S\ref{struct}, we put $\gn$ for the kernel of the (differential of the) very special isogeny in $\g = \Lie(G)$.
 
 \begin{lem} \label{ann}
 For any representation $V$ of any Lie algebra $\gd$, we have: $(V/V^{\gd})^\gd \subseteq V^{[\gd, \gd]} / V^\gd$.  In particular, if $\gd = [ \gd, \gd]$, then $(V/V^{\gd})^\gd = 0$.
 \end{lem}
 
 \begin{proof}
Suppose for some $v \in V$ that $zv \in V^\gd$ for all $z \in \gd$.  For $x, y \in \gd$, we have $[x,y]v = x(yv) - y(xv) = 0$, i.e., $v \in V^{[\gd, \gd]}$.
 \end{proof}

\begin{lem} \label{n.nvfree}
For $G$ simple and simply connected over a field $k$ such that $\car k$ is special, we have:
\begin{enumerate}
\item \label{n.vfree.vfree} $\gn$ is not virtually free as an $\gn$-module.  
\item \label{n.vfree.nz} $\gn / \gn^{[\g, \g]} \ne 0$ and $(\gn / \gn^{[\g, \g]})^{[\g, \g]} = 0$.
\end{enumerate}
\end{lem}

\begin{proof}  We first verify \eqref{n.vfree.vfree}.
If $G$ has type $G_2$ or $F_4$, then $\gn$ is the simple quotient of $\sl_3$ or $\spin_8$ by its center.  Thus $\gn$ acts on $\gn$ with trivial kernel and the stabilizer of a generic element of $\gn$ is the image of a maximal torus in $\sl_3$ or $\spin_8$, of dimension 1 or 2.

If $G$ has type $C_n$ with $n \ge 2$, then $\gn = [\so_{2n}, \so_{2n}]$, of codimension 1 in $[\g, \g] = \so_{2n}$.  Then $\gn^{[\g, \g]} = \gn \cap \z(\so_{2n})$ and the quotient $\gn / \gn^{[\g, \g]}$ is the irreducible representation of $\so_{2n}$ with highest weight the highest root.  Since the quotient is not virtually free for $\n$ as in \cite[Example 3.4]{GG:large}, neither is $\gn$ itself.   (Note that this also gives \eqref{n.vfree.nz} in this case.)

If $G$ has type $B_n$ with $n \ge 2$, then $\gn = \lsub$, the Heisenberg algebra from \S\ref{heis.sec}.
For a generic $h \in \lsub$, the map $x \mapsto [x,h] \in \z(\lsub)$ is a nonzero linear map, so has kernel of codimension 1, completing the proof of \eqref{n.vfree.vfree}.

Now consider \eqref{n.vfree.nz}. In the remaining cases where $G$ has type $B_n$ with $n \ge 3$, $G_2$, or $F_4$, the algebra $\g$ is perfect (immediately giving the second claim by Lemma \ref{ann}), so $\gn^{[\g,\g]} = \gn^\g = \gn \cap \z(\g)$, which is zero for type $G_2$ and $F_4$ and has codimension $2n$ in $\gn$ for type $B_n$.
\end{proof}

\begin{eg} \label{large.eg}
Let $U$ be a representation of $G$ on which $\g$ acts with kernel $\gn$.    For $G$ of type $B_n$ with $n \ge 2$, we take $U$ to be the natural irreducible representation of dimension $2n$.  For type $G_2$ and $F_4$, we take $U = \g/\gn$.  For type $C_n$ with $n \ge 3$, $\g/\gn$ is the Heisenberg subalgebra in $\spin_{2n+1}$ and we take $U$ to be the spin representation.  Note that for types $B_n$, $F_4$, and $G_2$, the algebra $\g$ is perfect and $U$ is irreducible so $U^{[\g, \g]} = 0$.   For type $C_n$, $[\g, \g]$ maps to the center of $\spin_{2n+1}$, and again we have $U^{[\g, \g]} = 0$.

Now take $W$ to be a finite sum of enough copies of $U$ so that $\g$ acts virtually freely on $W$, compare Example \ref{B.nat} for type $B_n$ and Example \ref{B.spin} for type $C_n$.  For types $F_4$ and $G_2$, take $W = U \oplus U$.

Let $V = \gn \oplus \bigoplus^m W$ for $W$ as in the preceding paragraph.  The subquotient $X = V/V^{[\g,\g]} = \gn/\gn^{[\g, \g]} \oplus \bigoplus^m W$ has, by Lemma \ref{n.nvfree}, $X^{[\g, \g]} = 0$ and $\dim X > m \dim W$ can be made arbitrarily large by increasing $m$.  On the other hand, for a generic vector $v = (n, w_1, \ldots, w_m) \in V$, the stabilizer $\g_v$ equals $\gn_n$, so by Lemma \ref{n.nvfree} $\g$ does not act virtually freely on $V$.
\end{eg}

\section{Proof of Theorem \ref{MT.special}} \label{MT.special.sec}

The goal of this section is to complete the proof of Theorem \ref{MT.special}, so we adopt its hypotheses.
In particular, $\car k$ is assumed to be special for $G$ and $\rho$ is faithful.
Note that $V^\gn = 0$; otherwise, as $V$ is irreducible, $V^\gn = V$ and $\drho$ would not be faithful.

Write the highest weight $\la$ of $\rho$ as $\la_0 + p \la_1$ where $\la_0, \la_1$ are dominant, $p = \car k$, and $\la_0$ is restricted.  Assume $\la_1 \ne 0$ for otherwise we are done by Theorem \ref{MT.restricted}, because the representations in Table \ref{irred.vfree} are not faithful.  As $G$ acts faithfully, it follows as in \cite[Lemma 1.1]{GG:irred} that $\la_0 \ne 0$ and $\dim V > \dim G$.  We will verify that $\g$ acts generically freely on $V$.

As $\la \in T^*$ and $p\la_1$ is in the root lattice hence in $T^*$, it follows that $\la_0$ is also in $T^*$.  Therefore, the representations $L(\la_0)$ and $L(p\la_1)$ are representations of $G$.  The representation $L(\la_0) \ot L(p\la_1)$ is a representation of $G$ that is irreducible (because its restriction to the simply connected cover of $G$ is the representation $L(\la_0) \ot L(\la_1)^{[p]}$) and has highest weight $\la$, so it is equivalent to $V$.

We will repeatedly use below that for any representation $X$ of $G$ and any $e > 0$, the representation $X \ot X^{[p]^e}$ is virtually free for $\g$, see \cite[Lemma 10.2]{GG:irred}.

\subsubsection*{Type $C$} Suppose first that $G$ has type $C_n$ for some $n \ge 3$.  The smallest nontrivial irreducible representation of $\Sp_{2n}$ is the natural representation of dimension $2n$.  If $\dim L(\la_0) = \dim L(p\la_1) = 2n$, then $G = \Sp_{2n}$ and $L(p\la_1)$ is a Frobenius twist of $L(\la_0)$, so $V$ is generically free as in the preceding paragraph.  Otherwise, at least one of $\dim L(\la_0)$, $\dim L(p\la_1)$ is gretater than $2n$.  The second smallest nontrivial irreducible representaiotn of $\Sp_{2n}$ is $L(\omega_{n-1})$ of dimension
\[
s(n) = \begin{cases}
2n^2 - n - 1 & \text{if $n$ is odd;} \\
2n^2 - n - 2 & \text{if $n$ is even.}
\end{cases}
\]
So $\dim V \ge 2n\,s(n)$, the values of which are as follows:
\[
\begin{array}{c|ccc}
n&3&4&\ge 5 \\ \hline
2n\,s(n)&84&208&>6n^2 + 6n
\end{array}
\]
Therefore, $G$ acts generically freely by Prop.~\ref{C.faithful}.

\subsubsection*{$G$ simply connected}
We are reduced to considering $G$ of type $B_n$ for $n \ge 2$, $F_4$, or $G_2$.  By Lemma \ref{SO.faithful}, $G$ is simply connected, and therefore $V \cong L(\la_0) \ot L(\la_1)^{[p]}$.  As a representation of $\g$, this is a sum of $\dim L(\la_1)$ copies of $L(\la_0)$, and in particular $L(\la_0)$ is itself faithful.  Put $m$ for the dimension of the smallest nontrivial irreducible representation of $G$, which is $2n$ (type $B_n$), 26 (type $F_4$), or 7 (type $G_2)$.  Then $V$ contains the $\g$-submodule $X := \oplus^m L(\la_0)$ on which $\g$ acts faithfully, and we will show that $\g$ acts generically freely on $X$.

If $\dim L(\la_0) = m$, then $X$ is isomorphic to $L(\la_0) \ot L(\la_0)^{[p]}$ as $\g$-modules, and we are done, so assume $\dim L(\la_0) > m$.   For $m'$ for the dimension of the second smallest nontrivial irreducible representation, we may assume that $\dim L(\la_0) \ge m'$, whence $\dim X \ge mm'$.  

If $G$ has type $F_4$ or $G_2$, then $m' = 246$ or 27 respectively, and Prop.~\ref{FG.dim} shows that $\g$ acts generically freely on $X$.

So suppose $G = \Spin_{2n+1}$ for some $n \ge 2$.  The smallest faithful irreducible representation of $G$ is the spin representation $L(\omega_1)$ of dimension $2^n$, so $\dim V \ge m2^n = n2^{n+1}$.  If $n \ge 4$, then $\dim V > 4n^2 + 4n$, and we are done by Lemma \ref{B.faithful}.
For $n = 2, 3$, a sum of $2n$ copies of the spin representation is generically free (Example \ref{B.spin}), so we may assume that $\la_0 \ne \omega_1$.  The next smallest faithful representation of $\Spin_{2n+1}$ is $L(\omega_1 + \omega_n)$ of dimension 16 for $n = 2$ or 48 for $n = 3$.  Therefore, $\dim X \ge 2n \cdot \dim L(\omega_1 + \omega_n) > 4n^2 + 4n$ and again we are done by Lemma \ref{B.faithful}, completing the proof of Theorem \ref{MT.special}.$\hfill\qed$

\section{Proof of Theorem \ref{MT.group}} \label{MT.group.sec}

Theorem \ref{MT.group} now follows quickly from what has gone before.  We repeat the argument given at the end of part I for the convenience of the reader.

The stabilizer $G_v$ of a generic $v \in V$ is finite \'etale if and only if the stabilizer $\g_v$ of a generic $v \in V$ is zero, i.e., if and only if $\g$ acts generically freely on $V$.  By Theorem A in \cite{GG:irred} (for which the case where $\car k$ is special is Th.~\ref{MT.special} in this paper), this occurs if and only if $\dim V > \dim G$ and $(G, \car k, V)$ does not appear in Table \ref{irred.nvfree}, proving Theorem \ref{MT.group}\eqref{MT.group.et}.

For Theorem \ref{MT.group}\eqref{MT.group.free}, we must enumerate in Table \ref{ngenfree} those representations $V$ such that  $\dim V > \dim G$, $V$ does not appear in Table \ref{irred.nvfree}, and the group of points $G_v(k)$ is not trivial.  Those $V$ with the latter property are enumerated in \cite{GurLawther}, completing the proof of Theorem \ref{MT.group}\eqref{MT.group.free}.$\hfill\qed$

\bibliographystyle{amsalpha}
\providecommand{\bysame}{\leavevmode\hbox to3em{\hrulefill}\thinspace}
\providecommand{\MR}{\relax\ifhmode\unskip\space\fi MR }
\providecommand{\MRhref}[2]{%
  \href{http://www.ams.org/mathscinet-getitem?mr=#1}{#2}
}
\providecommand{\href}[2]{#2}

\end{document}

%% file: ngenfree-lie-table-enhanced.tex
\begin{table}[htbp]
\begin{tabular}{cccrc||cccrc}
$G$&$\car k$&rep'n&$\dim V$&$\dim \g_v$&$G$&$\car k$&high weight&$\dim V$&$\dim \g_v$\\ \hline\hline
$\SL_8/\mu_4$&2&$\wedge^4$&70&3&$\Sp_8$&3&0100&40&2 \\
$\SL_9/\mu_3$&3&$\wedge^3$&84&2&$\Sp_4$&5&11&12&1\\
$\Spin_{16}/\mu_2$&2&half-spin&128&4&$\SL_4$&$p$ odd&$01p^e$, $e \ge 1$&24&1 \\
&&&&&$\SL_4/\mu_2$&2&$012^e$, $e \ge 2$&24&1
\end{tabular}
\caption{Irreducible and faithful representations $V$ with restricted highest weight of simple $G$ with $\dim V > \dim G$ that are not generically free for $\g$, up to graph automorphism.  For each, the stabilizer $\g_v$ of a generic $v \in V$ is a toral subalgebra.  The weights on the right side are numbered as in Table \ref{luebeck.table}.} \label{irred.nvfree}
\end{table}

%% file: ngenfree-table.tex
\begin{table}[htbp]
\begin{tabular}{cccr||cccr}
$G$&$\car k$&$V$&$\dim V$&$G$&$\car k$&$V$&$\dim V$\\ \hline\hline
$A_1$&$\ne 2, 3$&$S^3$&4&$A_2$&$\ne 2, 3$&$S^3$&10 \\
$A_1$&$\ne 2, 3$&$S^4$&5&$A_3$&$\ne 2$&$L(2\omega_2)$&19 or 20 \\
$A_8$&$\ne 3$&$\wedge^3$&84&$A_7$&$\ne 2$&$\wedge^4$&70 \\
$A_3$&3&$L(\omega_1+\omega_2)$&16&$A_\ell$&$p \ne 0$&$L(\omega_1 + p^i \omega_\ell)$,&$(\ell+1)^2$ \\
&&&&&&$L(\omega_1 + p^i \omega_1)$ \\ 
$B_\ell$ ($\ell \ge 2$)&$\ne 2$&$L(2\omega_\ell)$& $2\ell^2 - 3\ell - \eps$&$C_4$&$\ne 2$&``spin''&41 or 42 \\
$D_\ell$ ($\ell \ge 4$)&$\ne 2$&$L(2\omega_\ell)$&$2\ell^2 + \ell -1 - \eps$ &$D_8$&$\ne 2$&half-spin&128
\end{tabular}
\caption{Irreducible faithful representations $V$ of simple $G$ with $\dim V > \dim G$ such that $G_v$ is finite \'etale and $\ne 1$ for generic $v \in V$, up to graph automorphism, adapted from \cite{GurLawther}.  The symbol $\eps$ denotes 0 or 1 depending on the value of $\car k$.} \label{ngenfree}
\end{table}

%% file: special.bbl
\begin{thebibliography}{CGP15}

\bibitem[BCP90]{Magma}
W.~Bosma, J.~Cannon, and C.~Playoust, \emph{The {M}agma algebra system {I}: the
  user language}, J. Symbolic Computation \textbf{9} (1990), 677--698.

\bibitem[BT73]{BoTi:hom}
Armand Borel and Jacques Tits, \emph{Homomorphismes ``abstraits'' de groupes
  alg\'ebriques simples}, Ann. of Math. (2) \textbf{97} (1973), 499--571.

\bibitem[CGP15]{CGP2}
B.~Conrad, O.~Gabber, and G.~Prasad, \emph{Pseudo-reductive groups}, 2nd ed.,
  Cambridge University Press, 2015.

\bibitem[DG70]{DG}
M.~Demazure and P.~Gabriel, \emph{Groupes alg\'ebriques. {T}ome {I}:
  {G}\'eom\'etrie alg\'ebrique, g\'en\'eralit\'es, groupes commutatifs},
  Masson, Paris, 1970.

\bibitem[GG07]{GoGu}
D.~Goldstein and R.~Guralnick, \emph{Alternating forms and self-adjoint
  operators}, J. Algebra \textbf{308} (2007), no.~1, 330--349.

\bibitem[GG15]{GG:simple}
S.~Garibaldi and R.M. Guralnick, \emph{Simple groups stabilizing polynomials},
  Forum of Mathematics: Pi \textbf{3} (2015), e3 (41 pages).

\bibitem[GG17]{GG:spin}
\bysame, \emph{Spinors and essential dimension}, Compositio Math. \textbf{153}
  (2017), 535--556, with an appendix by A. Premet.

\bibitem[GG19a]{GG:large}
\bysame, \emph{Generically free representations {I}: Large representations},
  April 2019, arxiv:1711.05502v3.

\bibitem[GG19b]{GG:irred}
\bysame, \emph{Generically free representations {II}: {I}rreducible
  representations}, April 2019, arxiv:1711.06400v3.

\bibitem[GL19]{GurLawther}
R.M. Guralnick and R.~Lawther, \emph{Generic stabilizers in actions of simple
  algebraic groups {I}: modules and first {G}rassmannian varieties}, in
  preparation, 2019.

\bibitem[Hes79]{Hesselink}
W.H. Hesselink, \emph{Nilpotency in classical groups over a field of
  characteristic 2}, Math. Zeit. \textbf{166} (1979), 165--181.

\bibitem[His84]{Hiss}
G.~Hiss, \emph{Die adjungierten {D}arstellungen der {C}hevalley-{G}ruppen},
  Arch. Math. (Basel) \textbf{42} (1984), 408--416.

\bibitem[Hog82]{Hogeweij}
G.M.D. Hogeweij, \emph{Almost-classical {L}ie algebras. {I}, {II}}, Nederl.
  Akad. Wetensch. Indag. Math. \textbf{44} (1982), no.~4, 441--460.

\bibitem[Jan03]{Jantzen}
J.C. Jantzen, \emph{Representations of algebraic groups}, second ed., Math.
  Surveys and Monographs, vol. 107, Amer. Math. Soc., 2003.

\bibitem[LS12]{LiebeckSeitz}
M.~Liebeck and G.~Seitz, \emph{Unipotent and nilpotent classes in simple
  algebraic groups and {L}ie algebras}, Math. Surveys and Monographs, no. 180,
  Amer. Math. Soc., 2012.

\bibitem[L{\"u}b01]{luebeck}
F.~L{\"u}beck, \emph{Small degree representations of finite {C}hevalley groups
  in defining characteristic}, LMS J. Comput. Math. \textbf{4} (2001),
  135--169.

\bibitem[Pin98]{Pink:cpt}
R.~Pink, \emph{Compact subgroups of linear algebraic groups}, J. Algebra
  \textbf{206} (1998), no.~2, 438--504.

\bibitem[SF88]{StradeF}
H.~Strade and R.~Farnsteiner, \emph{Modular {L}ie algebras and their
  representations}, Monographs and textbooks in pure and applied math., vol.
  116, Marcel Dekker, New York, 1988.

\bibitem[Ste63]{St:rep}
R.~Steinberg, \emph{Representations of algebraic groups}, Nagoya Math. J.
  \textbf{22} (1963), 33--56, [= Collected Papers, pp.~149--172].

\bibitem[Var01]{Raja}
V.S. Varadarajan, \emph{${\Spin}(7)$-subgroups of ${\SO}(8)$ and ${\Spin}(8)$},
  Expo. Math. \textbf{19} (2001), 163--177.

\end{thebibliography}
